\documentclass[a4paper,11pt]{amsart}
\usepackage[T2A]{fontenc}
\usepackage[cp1251]{inputenc}%MS-DOS Windows code-page
\usepackage[dvips]{graphicx}
\usepackage{amsmath}
\usepackage{amsfonts}
\usepackage{amssymb}
\usepackage{amsthm}
\usepackage{euscript}
\author[E.A.~Makedonskyi]{Ievgen Makedonskyi}
\address{
Faculty of Mechanics and mathematics, Kyiv Taras Shevchenko
University\\
Volodymyrska str., 64, Kyiv, 01033, Ukraine}
\email{makedonskyi.e@gmail.com}
\title[On noncommutative bases of the free module  $W_n(\mathbb K)$
] {On noncommutative bases of the free module  $W_n(\mathbb K)$ of
all $\mathbb K$-derivations of the polynomial ring in $n$
variables }
\DeclareMathOperator\Der{Der}
\DeclareMathOperator\dive{div}
\DeclareMathOperator\ad{ad}
 \DeclareMathOperator\tr{tr}
\DeclareMathOperator\differential{d}
\renewcommand\d\differential

\let\leq\leqslant
\let\geq\geqslant

\let\star *

\let\supset\supseteq

\newtheorem{theorem}{Theorem}
\newtheorem{lemma}{Lemma}
\newtheorem{proposition}{Proposition}
\newtheorem{corollary}{Corollary}

\theoremstyle{definition}
\newtheorem{example}{Example}
\newtheorem{remark}{Remark}

\begin{document}

\sloppy

\begin{abstract}
 Let  $\mathbb{K}$ be an
algebraically closed field of  characteristic zero and
$R=\mathbb{K}[x_1,x_2,...x_n]$  the polynomial ring in  $n$
variables over $\mathbb K.$  We study bases of the free $R$-module
$W_n(\mathbb{K})$ of all  $\mathbb{K}$-derivations of the ring
$R$, such that their linear span over $\mathbb K$ is a subalgebra
of the Lie algebra  $W_n(\mathbb{K})$. We proved that for any Lie
algebra  $L$ of dimension  $n$ over $\mathbb{K}$ there exists a
subalgebra  $\overline{L}$ of $W_n(\mathbb{K})$ which is
isomorphic to $L$ and such that every  $\mathbb{K}$-basis of
$\overline L$ is an $R$-basis of the $R$-module
$W_n(\mathbb{K})$.

\end{abstract}

\maketitle

\section{Introduction}

    Let  $\mathbb{K}$ be an algebraically closed field of characteristic $0$
    and $W_n(\mathbb{K})$ the  Lie algebra of all  $\mathbb
    K$-derivations of the polynomial  $R=\mathbb{K}[x_1,x_2,...x_n]$ in  $n$
    variables over $\mathbb K.$
       The structure of the Lie algebra $W_n(\mathbb{K})$ and its subalgebras was studied by many authors
        (see, for example, \cite{Buh,Jordan,Petr}).
     One of the most important problem here is the question about
     structure of finite dimensional subalgebras of  $W_n(\mathbb{K})$.
      The description of all such subalgebras in  $W_1(\mathbb{K})$ and  $W_2(\mathbb{K})$
      in case of the field  $\mathbb K=\mathbb C$ of complex numbers can be easily
      obtained from the works of S.Lie \cite{Lie}.
      There is no such a description for the Lie algebra $W_n(\mathbb{K}), \ n\geq 3,$
      so it is of interest to study some classes of finite
      dimensional subalgebras of the algebra $W_n(\mathbb{K}).$

      Define one of such classes in the following way: a
      subalgebra $L$ of $W_n(\mathbb{K})$ of dimension $n$
      over $\mathbb{K}$ will be called {\emph{ basic,}} if every basis of
      the algebra $L$ over $\mathbb{K}$ is a basis of the
      $R$-module $W_n(\mathbb{K})$. It is obvious, that the Lie
      algebra $\mathbb K\langle \frac{\partial }{\partial x_1},\frac{\partial }{\partial x_2},..
      .\frac{\partial }{\partial x_n} \rangle$  is abelian and basic.
      All abelian basic Lie subalgebras in $W_n(\mathbb{K})$ are
      described in  \cite{Now1}  (see also \cite{Now}): let  $f_1,f_2,...f_n \in R$
       such that  $\det J(f_1,f_2,...f_n)\in {\mathbb{K}}^\star$, where $J(f_1,f_2,...f_n)$
       is the Jacoby matrix of polynomials   $ f_1,f_2,...f_n \in R.$ Then the derivations
        $D_{1}, \ldots ,D_{n}\in W_n(\mathbb{K})$ defined by the
        conditions
         $D_i(h)=\det J(f_1,...f_{i-1},h,f_{i+1},...f_n)$
        for any $h\in R$ pairwise commute and form a basis of the  $R$-module
         $W_n(\mathbb{K})$.  Conversely, every basis $\lbrace D_1,D_2,...D_n\rbrace$ such that
          $[D_i,D_j]=0, \ i, j=1, \ldots ,n$,
        can be obtained in such a way. But there exist also non-abelian basic
        subalgebras.
        For example, the two-dimensional subalgebra with a
        basis $\lbrace \frac{\partial}{\partial x_1}+x_2\frac{\partial}{\partial x_2},
        \frac{\partial}{\partial x_2}\rbrace$ from  $W_2(\mathbb{K})$ is basic and non-abelian.
     The main result of the paper, Theorem \ref{mainth}, shows that
     every Lie algebra of dimension $n$ over  $\mathbb K$ is isomorphic to a basic subalgebra
     in  $W_n(\mathbb{K})$. In Theorem \ref{nil}, all nilpotent basic
     subalgebras of the Lie algebra $W_n(\mathbb{K})$ are
     characterized.

     We use standard notation in the paper  (see, for example, \cite {Now}).
    The ground field  $\mathbb K$ is algebraically closed of characteristic $0.$
   The basis  $\lbrace \frac{\partial} {\partial x_1},\frac{\partial}{\partial x_2},...\frac{\partial
    }{\partial x_n} \rbrace$ of the free $R$-module $W_n(\mathbb{K})$ will be called standard.
    Every another basis  $\lbrace D_1,D_2, \ldots , D_n\rbrace$ can be obtained from the standard
    one by an invertible matrix  $A=\left(
                       \begin{array}{cccc}
                         a_{11} & a_{12} & \ldots & a_{1n} \\
                         a_{21} & a_{22} & \ldots & a_{2n} \\
                         \vdots & \vdots & \ddots & \vdots \\
                         a_{n1} & a_{n2} & \ldots & a_{nn} \\
                       \end{array}
                     \right)
$, where $D_i=a_{i1}\frac{\partial }{\partial x_1}+\ldots
+a_{in}\frac{\partial }{\partial x_n}$, $i=1, \ldots , n.$ If
$w_{1}, \ldots , w_{n}\in R=\mathbb K[x_{1}, \ldots , x_{n}]$ then
by $J(w_{1}, \ldots , w_{n})$ will be denoted the Jacoby matrix
 of these polynomials.

\section{On nilpotent basic subalgebras}

\begin{proposition} \label{qq}
 Let $L$ be a basic  Lie subalgebra in $W_n(\mathbb{K})$ and  $d$  an element of $L$.
 Then the trace of $d$  in $L$ by the  adjoint representation
  is equal to its divergence taken with the opposite sign, i.e. $\tr d = -\dive
  d.$
\end{proposition}

\begin{proof}
Let $\lbrace D_1,D_2,...D_n\rbrace$ be an arbitrary basis of the
Lie algebra $L$ over $\mathbb{K}$. Let $D_i=a_{i1}\frac{\partial
}{\partial x_1}+\ldots +a_{in}\frac{\partial }{\partial x_n}$ be a
decomposition of $D_i$ in standard basis, $a_{ij}\in
R=\mathbb{K}[x_1,x_2,...x_n]$. Since $\lbrace
D_1,D_2,...D_n\rbrace$ is a basis of the free $R$-module
$W_n(\mathbb{K})$, it holds obviously $\Delta =\det (a_{ij}) \in
\mathbb{K}^\star$. By Lemma 2.3 from \cite {Buh}, the divergences
${\dive} D_i=\sum_{j=1}^{n}{\frac{\partial a_{ij}}{ {\partial
x_j}}}$ and traces ${\rm{tr}}( {\ad}
D_i)=\sum_{j=1}^{n}{c_{ij}^j}$ satisfy the following relation:
${\rm{tr}}( {\ad} D_i)=D_i(\Delta)- {\dive} D_i.$ Since
$D_{i}(\Delta )=0$ we have  ${\rm{tr}}( {\ad} D_i)=-{\dive} D_i$.
The element $d$ is a linear combination of $D_{i}$ with
coefficients in $\mathbb K$ therefore $\tr d = -\dive
  d.$
\end{proof}

\begin{corollary}
If $L$ is a semisimple or nilpotent basic subalgebra of
$W_n(\mathbb{K})$ then  $L \in SW_n(\mathbb{K})$, where
$SW_n(\mathbb{K})$ is the Lie algebra of all  $D\in
W_n(\mathbb{K})$ such that $\dive D =0$.
\end{corollary}

\begin{lemma}~\label{main}
Let $\lbrace D_1,D_2,...D_n\rbrace$ be an arbitrary basis of the
$R$-module $W_n(\mathbb{K})$ and  $\left[D_i,D_j
\right]=\sum_{k=1}^{n}{c_{ij}^kD_k}$ for some $c_{ij}^k\in
R=\mathbb{K}[x_1,x_2,...x_n]$. Write down  $\frac{\partial
}{\partial x_i}=\sum_{j=1}^n {b_{ij}D_j}$ for all $i=1, \ldots n$,
where $b_{ij} \in R$. Then

\begin{equation}
\sum_{i,j=1}^n {b_{pi}b_{qj}c_{ij}^k}+\frac{\partial
b_{qk}}{\partial x_p}-\frac{\partial b_{pk}}{\partial x_q}=0
\label{mainequ}
\end{equation}
for any $p,q,k=1,\ldots ,n$.

    Conversely, let $\lbrace D_1,D_2,...D_n\rbrace$ be a basis of the $R$-module $W_n(\mathbb{K})$.
    Define the elements $b_{ij}$ from the equations
    $\frac{\partial }{\partial x_i}=\sum_{j=1}^n {b_{ij}D_j}$.
    If these elements satisfy relations ~(\ref{mainequ}) for some $c_{ij}^k  \in R$, then $\left[D_i,D_j
     \right]=\sum_{k=1}^{n}{c_{ij}^kD_k}$.

\end{lemma}

\begin{proof}
Using commutativity of the standard basis, we get:
\begin{equation}
0=\frac{\partial }{\partial x_p}\frac{\partial }{\partial x_q}-\frac{\partial
}{\partial x_q}\frac{\partial }{\partial x_p}=
\frac{\partial }{\partial x_p} \sum_{j=1}^n {b_{qj}D_j} -\frac{\partial
}{\partial x_q} \sum_{j=1}^n {b_{pj}D_j}
\label{q2}
\end{equation}
for any  $p,q=1,\ldots ,n$.

Observe that $\frac{\partial }{\partial x_p} \sum_{j=1}^n
{b_{qj}D_j}=
 \sum_{j=1}^n {b_{qj} \frac{\partial }{\partial
 x_p} D_j}+\sum_{j=1}^n {\frac{\partial b_{qj}}{\partial x_p} D_j},$ analogously,
  $\frac{\partial} {\partial x_q} \sum_{j=1}^n {b_{pj}D_j}=
 \sum_{j=1}^n {b_{pj} \frac{\partial}{\partial x_q} D_j}+\sum_{j=1}^n {\frac{\partial b_{pj}}{\partial x_q }D_j}$. If we
 combine this with ~(\ref{q2}), we have the relation
 \[0=\sum_{j=1}^n {b_{qj} \frac{\partial }{\partial x_p} D_j}-\sum_{i=1}^n {b_{pi}
 \frac{\partial }{\partial x_q} D_i}+\sum_{k=1}^n{\left(\frac{\partial b_{qk} }{\partial x_p} -\frac{\partial b_{pk} }{\partial x_q }\right)D_k}.\]
 Substituting $\sum_{j=1}^n {b_{pj}D_j}$ and  $\sum_{j=1}^n {b_{qj}D_j}$
 instead of $\frac{\partial }{\partial x_p}$ and respectively  $\frac{\partial }{\partial x_q}$,
 we get the equality:
  \[0=\sum_{j=1}^n {b_{qj} \left(\sum_{i=1}^n b_{pi}D_i\right) D_j}-\sum_{i=1}^n {b_{pi}
  \left(\sum_{j=1}^n b_{qj}D_j\right) D_i}+\sum_{k=1}^n{\left(\frac{\partial b_{qk} }{\partial x_p}
   -\frac{\partial b_{pk} }{\partial x_q} \right)D_k}.\]
It is easy to see that
\[ \sum_{j=1}^n {b_{qj} \left(\sum_{i=1}^n b_{pi}D_i\right) D_j}-\sum_{i=1}^n {b_{pi}
  \left(\sum_{j=1}^n b_{qj}D_j\right) D_i}=\sum _{i,
  j=1}^{n}b_{pi}b_{qj}\ [D_{i}, D_{j}]. \]

 Combining this with the decomposition $\left[D_i,D_j \right]=\sum_{k=1}^n c_{ij}^k D_k$,
 we obtain the  relation:
\[ \sum_{i,j,k=1}^n {b_{pi}b_{qj}c_{ij}^k}D_k+\sum_{k=1}^n({\frac{\partial
b_{qk}}{\partial x_p}-\frac{\partial b_{pk}}{\partial
x_q}})D_{k}=0.
\]

The derivations $\lbrace D_1,D_2,...D_n\rbrace$ are linearly
independent over $R$, therefore we have relation  (1). This
completes the proof of the first part of our statement.

To prove the second part of the Lemma we define elements $b_{ij} $
by the next relations:
\[\frac{\partial
}{\partial x_i}=\sum_{j=1}^n {b_{ij}D_j}\ \ , i, j=1, \ldots ,
n.\] Suppose the elements $c_{ij}^k$ satisfy relations
~(\ref{mainequ}). We have $\left[D_i,D_j \right]=\sum
_{k=1}^{n}{\gamma_{ij}^k D_k}$ for some $\gamma_{ij}^k\in R$,
because $\lbrace D_1,D_2,...D_n\rbrace$ is a basis of the
$R$-module $W_n(\mathbb{K}).$ By the first part of this Lemma the
elements $\gamma_{ij}^k$ satisfy the relations ~(\ref{mainequ}),
that is, $\sum_{i,j=1}^n
{b_{pi}b_{qj}\gamma_{ij}^k}+\frac{\partial b_{qk}}{\partial
x_p}-\frac{\partial b_{pk}}{\partial x_q}=0$. The system
~(\ref{mainequ})  can be regarded as a linear system in $n^{3}$
variables $c_{ij}^k$. This system can be decomposed into a direct
sum of $n$ linear systems:
\begin{equation}  \sum_{i,j=1}^n {b_{pi}b_{qj}c_{ij}^k}+\frac{\partial
b_{qk}}{\partial x_p}-\frac{\partial b_{pk}}{\partial x_q}=0, \ \
k \  \mbox{is fixed,} \ \  k=1 \ldots n. \  \label{equk}
\end{equation}
 Let us prove that the system  (\ref{equk}) has a unique solution. It
is easy to see, that this system has the following matrix
                    $$\left(
                       \begin{array}{ccccccccc}
                         b_{11}b_{11} & b_{11}b_{12} & \ldots & b_{11}b_{1n}& \ldots & b_{1n}b_{11} & b_{1n}b_{12} & \ldots & b_{1n}b_{1n}\\
                         b_{11}b_{21} & b_{11}b_{22} & \ldots & b_{11}b_{2n}& \ldots &b_{1n}b_{21} & b_{1n}b_{22} & \ldots & b_{1n}b_{2n}\\
                         \vdots & \vdots & \ddots & \vdots & & \vdots & \vdots & \ddots & \vdots \\
                         b_{11}b_{n1} & b_{11}b_{n2} & \ldots & b_{11}b_{nn}& \ldots &b_{1n}b_{n1} & b_{1n}b_{n2} & \ldots & b_{1n}b_{nn}\\
                         \vdots & \vdots &  & \vdots & & \vdots& \vdots &  & \vdots \\
                         b_{n1}b_{n1} & b_{n1}b_{n2} & \ldots & b_{n1}b_{nn}& \ldots &b_{nn}b_{n1} & b_{nn}b_{n2} & \ldots & b_{nn}b_{nn}\\
                       \end{array}
                     \right).$$
                     Because this  matrix is obviously the  tensor
                     square $\left(b_{ij}\right)\otimes\left(b_{ij}\right)$ of the matrix
$\left(b_{ij}\right)$ and  the  determinant $\det (b_{ij})$
 is invertible (because $(b_{ij})$ is a transition matrix between two bases),
it holds
$\det\left(\left(b_{ij}\right)\otimes\left(b_{ij}\right)\right)=(\det
(b_{ij}))^{2n}\in \mathbb K^{\star}$.  Therefore, the system
(\ref{equk}) has the unique solution $\gamma_{ij}^k =c_{ij}^k$
\end{proof}
Now let $L$ be an arbitrary $n$-dimensional Lie algebra over the
field $\mathbb{K}$. Take a basis   $\{ l_1, \ldots , l_n \}$  of
algebra $L$ and  denote by $c_{ij}^k$  the structure constants of
$L$ in this basis, that is
 $\left[l_i,l_j\right]=\sum_{k=1}^n{c_{ij}^k l_k}$.
It is well known that the tensor product  $R\otimes _{\mathbb K}L$
of an associative and commutative $\mathbb K$-algebra $R$ and the
Lie algebra $L$ is a Lie algebra over the field $\mathbb K.$
Further,  we will always denote  by $R$ the polynomial algebra
$\mathbb K[x_{1}, \ldots , x_{n}].$ Since the elements of the
algebra $R\otimes _{\mathbb K}L$ are of the form
$\sum_{i=1}^n(f_i\otimes l_i), $ $f_{i}\in R, \ i=1, \ldots ,n,$
the tensor product  $R\otimes _{\mathbb K}L$ is a free module of
rank $n$ over the ring $R$. The elements $\{ 1\otimes l_{1},
\ldots , 1\otimes l_{n}\}$ form obviously a basis of this module.
Using the multiplication law in $L$, we get the equality
\begin{equation}
\left[\overline{f},\overline{g}\right]=\sum_{k=1}^n
{\left(\sum_{i,j=1}^n{c_{ij}^k f_i g_j}\right)\otimes l_k}.
\label{defmult}
\end{equation}
for any elements
 $\bar{f}=\sum_{i=1}^n f_i\otimes l_i, \ \bar{g}=\sum_{i=1}^n g_i\otimes l_i
 \in R\otimes _{\mathbb K}L.$

For an arbitrary element $\overline{f}=\sum_{i=1}^n f_i\otimes l_i
\in R\otimes _{\mathbb K}L$  and an arbitrary index $p=1, \ldots ,
n$ define $\frac{\partial \overline{f}}{\partial x_p}=\sum_{i=1}^n
{\frac{\partial f_i}{\partial x_p}\otimes l_i}$. It is easy to see
that the map $\overline{f}\mapsto \frac{\partial
\overline{f}}{\partial x_p}$ is a derivation of the Lie algebra
$R\otimes _{\mathbb K}L$ (we will  denote this map also  by
$\frac{\partial }{\partial x_p}$). Since the derivation
$\frac{\partial }{\partial x_p}$ acts on the coordinates $f_{i}$
of the element $\overline{f}=\sum_{i=1}^n f_i\otimes l_i$, it
holds $\frac{\partial }{\partial x_p}\frac{\partial }{\partial
x_q}-\frac{\partial }{\partial x_q}\frac{\partial }{\partial
x_p}=0$ for arbitrary $p,q=1, \ldots ,n.$
  Denote by $A$ the abelian Lie subalgebra  of $\Der (R\otimes _{\mathbb K}L)$ with
the basis $\lbrace \frac{\partial }{\partial x_1},\frac{\partial
}{\partial x_2},...\frac{\partial }{\partial x_n} \rbrace$  and by
$\widehat L$ the  subalgebra $\widehat L=A+R\otimes _{\mathbb K}L$
of the semidirect product of Lie algebras $\Der (R\otimes
_{\mathbb K}L)\rightthreetimes R\otimes _{\mathbb K}L$.

\begin{remark}\label{equiv}

Let $L$ be a basic subalgebra of the Lie algebra
$W_n(\mathbb{K})$. Then the equations (\ref{mainequ}) are
equivalent to the following relations in the Lie algebra $\widehat
L$:
\begin{equation}
\left[\overline{b_p},\overline{b_q}\right]+\left[
\frac{\partial}{\partial x_p},
\overline{b_q}\right]-\left[\frac{\partial}{\partial x_q},
\overline{b_p}\right]=0. \label{main1}
\end{equation}

Since $\left[\frac{\partial}{\partial
x_p},\frac{\partial}{\partial x_q}\right]=0$,  we can rewrite
relations (\ref{main1}) as the following relations in the Lie
algebra $\widehat L$
\begin{equation}
\left[\frac{\partial}{\partial x_p}+\overline{b_p},
\frac{\partial}{\partial x_q}+\overline{b_q}\right]=0.
\label{main2}
\end{equation}

\end{remark}
Let $L$ be a nilpotent Lie algebra over the field $\mathbb K$ with
$\dim L_{\mathbb K}=n.$  By Engel's theorem, $L$ has a flag of
ideals $L=L_0\supset L_1\supset...\supset L_{n-1}\supset
L_n=\lbrace 0 \rbrace$. Take any elements  $l_i \in
L_{i-1}\backslash L_i, \  i=1, \ldots , n$ and consider the Lie
algebra $\widehat L$ constructed in such a way as it was mentioned
above.
 The structure constants of $L$ in the basis
$\{ l_{1}, \ldots , l_{n}\}$ satisfy the relations
\begin{equation}
 c_{ij}^k=0, \ \mbox{if}  \ k\leq\max(i,j)\label{n}.
 \end{equation}

Since the Lie algebra $L$ is nilpotent, the tensor product
$R\otimes _{\mathbb K}L$ is also nilpotent. Then for any element
$\overline{w}\in R\otimes _{\mathbb K}L$ the inner derivation
${\ad} \overline{w}$ of the Lie algebra $\widehat L$ is nilpotent.
We collect some properties of the Lie algebras  $R\otimes
_{\mathbb K}L$  and  $\widehat L$ in the following Lemma.

\begin{lemma}\label{properties}
Let $L$ be a nilpotent  Lie algebra of dimension $n$ over the
field $\mathbb K$. Let $\{ l_{1}, \ldots , l_{n}\}$ be a basis of
$L$ and  $\overline{w}=\sum_{i=1}^n {w_{i}\otimes l_i}$ be an
arbitrary element of $R\otimes _{\mathbb K}L$. Then it holds

{\rm (1)}  $\ad \overline{w}$ is a nilpotent endomorphism of the
$R$-module $R\otimes _{\mathbb K}L$, that is $\ad
\overline{w}(f\overline u)=f\ad \overline{w}(\overline u)$ for any
$f\in R$ and $\overline u=\sum_{i=1}^n {u_{i}\otimes l_i}\in
R\otimes _{\mathbb K}L$;

{\rm (2)} $\varphi =\sum_{i=1}^{\infty}\frac{1}{i!}
\left({\ad} \
    \overline w\right)^{i-1}$  is an automorphism of the $R$-module
$R\otimes _{\mathbb K}L$;

{\rm (3)} if $\overline{w}=\sum_{i=1}^n {w_{i}\otimes l_i}$ is an
element of $R\otimes _{\mathbb K}L$ with the property $\det
J(w_{1}, \ldots , w_{n})=c\in \mathbb{ K}^{\star}$, then the set
of elements
$$\overline{b_p}=\sum_{i=1}^{\infty}{\frac{1}{i!}
    \left({\ad} \
    \overline w\right)^{i-1}\left(\frac{\partial w}{\partial x_p}\right)}, \ p=1, \ldots
    ,n$$
    is a basis of the free $R$-module $R\otimes _{\mathbb K}L.$
    In particular, writing $\overline{b_{p}}=\sum_{i=1}^n {b_{pi}\otimes l_i}$
    we have $\det (b_{ij})_{i,j=1}^n=c\in     \mathbb{K}^{\star}.$
 \end{lemma}
 \begin{proof} (1) Obvious.

(2) As $\ad \overline{w}$ is a nilpotent endomorphism of
$R$-module $R\otimes _{\mathbb K}L,$ the map $\varphi
=\sum_{i=1}^{\infty}\frac{1}{i!} \left(\ad \overline
w\right)^{i-1}$  is well defined and is an endomorphism of the
$R$-module $R\otimes _{\mathbb K}L.$ Since $\varphi
=E+\sum_{i=2}^{\infty}\frac{1}{i!} \left(\ad \overline
w\right)^{i-1}$ is the sum of the identity and a nilpotent
endomorphisms, the map $\varphi$ is an automorphism of the free
$R$-module $R\otimes _{\mathbb K}L.$

(3) Let $\overline{w}=\sum_{i=1}^n {w_{i}\otimes l_i}$ be an
element of $R\otimes _{\mathbb K}L$ such that  $\det J(w_{1},
\ldots , w_{n})=c\in \mathbb{ K}^{\star}.$ It is easy to see that
the set of elements $\{ \frac{\partial w}{\partial x_1}, \ldots ,
\frac{\partial w}{\partial x_n}\}$ is a basis of the $R$-module
$R\otimes _{\mathbb K}L.$ Since the  map $\varphi $ defined above
is an automorphism of the $R$-module $R\otimes _{\mathbb K}L$, the
set of elements $\overline{b_{p}}=\varphi (\frac{\partial
w}{\partial x_p}), \ p=1, \ldots ,n$ is the basis of this module.
Therefore $\det (b_{ij})_{i,j=1}^n=c\in \mathbb{K}^{\star}.$
 \end{proof}

\begin{theorem}~\label{nil}
{\rm{(1)}} Let $L$ be an arbitrary nilpotent Lie algebra over any
field  $\mathbb{K}$ of characteristic $0$. Then there exists a
basic subalgebra $\overline{L}$ of $W_n(\mathbb{K})$, such that
$\overline{L}$ is isomorphic to $L$ (by that  every basis of
$\overline L$ over $\mathbb K$ is a basis of $R$-module
$W_n(\mathbb{K})$).

    {\rm{(2)}} Let $\overline{L}$ be a basic Lie subalgebra of $W_n(\mathbb{K})$,
    such that $\overline{L}$ is isomorphic to a nilpotent Lie algebra
    $L$ with a basis
    $\lbrace l_1,l_2,...l_n\rbrace$, satisfying  the relations (\ref{n}), let $D_i$ be  the images of
    the elements $l_i$ by this isomorphism. Write down
    $\frac{\partial }{\partial x_i}=\sum_{j=1}^n {b_{ij}D_j}$,
    $\overline{b_p}=\sum_{i=1}^n {b_{pi}\otimes l_i}$, $p=1,\ldots ,n$.
    Then there exists an element $\overline{w}=\sum_{i=1}^n {w_{i}\otimes l_i}\in R\otimes _{\mathbb K}L$
    such that $\det J\left(w_{1},w_{2},...w_{n}\right) \in \mathbb{K}^\star$
    and the following equalities  hold:
    \[\overline{b_p}=\sum_{i=1}^{\infty}{\frac{1}{i!}
    \left({\ad} \
    \overline w\right)^{i-1}\left(\frac{\partial w}{\partial x_p}\right)}, \ p=1, \ldots ,n\]
    (the number of  nonzero summands in this series is finite
    because the Lie algebra $R\otimes _{\mathbb K}L$ is nilpotent).
    \end{theorem}

\begin{proof}
(1) Let $L$ be a nilpotent  Lie algebra over the field
$\mathbb{K}$ of dimension $n$ and $\lbrace l_1,l_2,...l_n\rbrace$
be its basis
 such that the structure constants $c_{ij}^k$ satisfy the relations (\ref{n}). We will
construct $b_{ij}\in R, i,j=1,\ldots n,$  which satisfy relations
(\ref{mainequ}) and have property $\det (b_{ij})_{i,j=1}^{n} \in
\mathbb{K}^\star$.  If we consider the Lie algebra $\widehat L$,
then by Remark \ref{equiv} the conditions (\ref{mainequ}) are
equivalent to the conditions (\ref{main2}) in the terms of
$\widehat L$ for $\overline{b_p}=\sum_{i=1}^n {b_{pi}\otimes
l_i}$, $p=1,\ldots ,n$:
\[\left[\frac{\partial}{\partial x_p}+\overline{{b_p}},
\  \frac{\partial}{\partial x_q}+\overline{{b_q}}\right]=0.\] Take
an arbitrary element $\overline{w}=\sum_{i=1}^n {w_{i}\otimes l_i}
\in \widehat{L}$. Put
$\overline{b_p}=\sum_{i=1}^{\infty}{\frac{1}{i!} \left({\ad}
\hskip 0.5mm \overline w\right)^{i-1}\left(-\frac{\partial
\overline w}{\partial x_p}\right)}, \ p=1, \ldots n,$  where as
usually $({\ad}
 \overline {w})^{0}=E$ is the identity
operator. Note, that the  relations
$$\sum_{i=1}^{\infty}{\frac{1}{i!} \left({\ad}
\overline w\right)^{i-1}\left(-\frac{\partial \overline
w}{\partial x_p}\right)}= \sum_{i=1}^{\infty}{\frac{1}{i!}
\left({\ad}  \overline
w\right)^{i-1}\left[\overline w,\frac{\partial}{\partial
x_p}\right]}= \sum_{i=1}^{\infty}{\frac{1}{i!} \left({\ad}
 \overline w\right)^{i}\left(\frac{\partial}{\partial
x_p}\right)}$$ hold in the Lie algebra $\widehat L$ for $p=1,
\ldots ,n.$ Therefore, we have the following equalities in the Lie
algebra $\widehat L$:
\[\left[\frac{\partial}{\partial x_p}+\overline{b_p},\frac{\partial}{\partial
 x_q}+\overline{b_q}\right]=\]
\[\left[\frac{\partial}{\partial x_p}+\sum_{i=1}^{\infty}{\frac{1}{i!}
\left({\ad}  \overline
w\right)^{i}\left(\frac{\partial}{\partial x_p}\right)},
\frac{\partial}{\partial x_q}+\sum_{i=1}^{\infty}{\frac{1}{i!}
\left({\ad}  \overline
w\right)^{i}\left(\frac{\partial}{\partial x_q}\right)}\right]=\]
\[\left[ e^{{\ad}  \hskip 0.5mm \overline w}\left(\frac{\partial}{\partial x_p}\right),
e^{{\ad}   \overline w}\left(\frac{\partial}{\partial x_q}\right)
\right]= e^{{\ad} \hskip 0.5mm \overline
w}\left[\frac{\partial}{\partial x_p},\frac{\partial}{\partial
x_q} \right]=0\]
 (we took into account that  $e^{{\ad}  \overline w}$ is
 an automorphism of the Lie algebra $\widehat L$ and $\left[\frac{\partial}{\partial
  x_q},\frac{\partial}{\partial x_p}\right]=0, \ p, q=1, \ldots,
  n$).

 Thus, we have  elements $\overline{b_1}=\sum_{i=1}^n {b_{1i}\otimes
 l_i}, \
  \overline{b_2}=\sum_{i=1}^n {b_{2i}\otimes l_i}, \ldots , \overline{b_n}=\sum_{i=1}^n {b_{ni}\otimes
  l_i}$
  satisfying equations (\ref{main1}), then these elements satisfy also
  the equations (\ref{mainequ}).

  Take any element   $\overline w$ of $R\otimes L$ such that $\det J(\overline w)\in \mathbb{K}^{\star},$
   for example, $\overline{w}=\sum_{i=1}^n {x_{i}\otimes l_i} \in \widehat{L}$.
   Let    $B^{-1}=A=(a_{ij})$ be the inverse matrix for $B$ (the
   matrix $B=(b_{ij})$ is invertible by  Lemma \ref{properties} p.3). By Lemma \ref{main},
  the derivations $D_i=\sum_{i=1}^n a_{ij} \frac {\partial}  {\partial x_j}$,$i=1,...n$,
  satisfy the  relations $\left[D_i,D_j\right]=\sum_{k=1}^n{c_{ij}^k}D_k$ and $\det (a_{ij})
   \in \mathbb{K}^\star$. We have shown that $\mathbb {K}\langle D_1,D_2,...D_n \rangle$ is a
   basic    subalgebra of $W_n(\mathbb{K})$, which is isomorphic to $L$. This completes the proof of part (1).

   (2) Take any elements
  $$\overline{b_1}=\sum_{i=1}^n {b_{1i}\otimes
  l_i},\
  \overline{b_2}=\sum_{i=1}^n {b_{2i}\otimes l_i} , \ldots , \overline{b_n}=\sum_{i=1}^n
  {b_{ni}\otimes l_i}$$
    of $R \otimes _{\mathbb K}L$  satisfying the relations (\ref{main2}).
    We will build an element  $\overline w\in R \otimes _{\mathbb K}L$
     such that $\overline{b_p}=\sum_{i=1}^\infty {\frac{1}{i!}( \ad \overline w )^{i-1}(-\frac{\partial w}{\partial
     x_p}}) $, $p=1, \ldots n$.
     This equality is equivalent to the relation
     $e^{\ad \overline w}\left(\frac{\partial}{\partial x_p}\right)=\frac{\partial }{\partial x_p}+\overline{b_p}$.
     Applying the automorphism $e^{\ad \overline {-w}}$ to the both sides of this relation, we obtain
     \begin{equation}
    \frac{\partial}{\partial x_p}=e^{\ad -\overline
    {w}}\left(\frac{\partial }{\partial x_p}+\overline{b_p}\right). \label{w}
    \end{equation}
  Thus, we must prove that this equality holds for $p=1, \ldots n$.
  Consider the  matrix
  $B=\left(\begin{array}{ccc} b_{11} & \ldots & b_{1n} \\
    \vdots & \ddots & \vdots \\
     b_{n1} & \ldots & b_{nn} \\
      \end{array} \right)$ with entries that are coordinates
      of the elements $\overline {b_1}, \ldots \overline {b_n}$.
  Let $m_1$ be the number of the first nonzero column of $B$
  (for $B=0$ take $w=0$, then (\ref{w}) obviously holds).
  Write the equations (\ref{mainequ}) for $k=m_1$:
  \begin{equation}
\sum_{i,j=1}^n {b_{pi}b_{qj}c_{ij}^{m_1}}+\frac{\partial
b_{q{m_1}}}{\partial x_p}-\frac{\partial b_{p{m_1}}}{\partial
x_q}=0, \  p,q=1, \ldots , n. \label{mainequm1}
\end{equation}
 Since  $m_1$ is the number of the first nonzero column of $B$, we have $b_{pi}=0$ for $i<m_1$, $p=1, \ldots n$.
 On the other hand, if $i \geq m_1$, then $c_{ij}^k=0$ by (\ref{n}) (because $L$ is nilpotent).
Hence (\ref{mainequm1}) is equivalent to the  equations
 $\frac{\partial b_{q{m_1}}}{\partial x_p}-\frac{\partial b_{p{m_1}}}{\partial x_q}=0$, $p,q=1, \ldots n$.
Then there exists a polynomial $h_1$  such that
 $ b_{p{m_1}}=\frac{\partial h}{\partial x_p}$, $p=1, \ldots n$ (see, for example
\cite{Now}).
 Denote   $\overline {h_1}=h_1 \otimes l_{m_1}$ and consider the elements
  $e^{\ad \overline h_1}\left(\frac{\partial }{\partial x_p}+\overline{b_p}\right) \ p=1, \ldots
  ,n.$
  Note that $$\overline {b_p}=\sum_{i=m_1}^n b_{pi} \otimes l_i=
  \frac{\partial h}{\partial x_p}\otimes l_{m_1}+\sum_{i=m_1+1}^n b_{pi} \otimes l_i,  \ p=1, \ldots
  n,$$ and therefore it holds
  $$e^{\ad \overline h_1}\left(\frac{\partial }{\partial x_p}+\overline{b_p}\right)=
  e^{\ad \overline h_1}\left(\frac{\partial }{\partial x_p}\right)+ e^{\ad \overline h_1}
  \left(\overline{b_p}\right)=$$ $$
  \frac{\partial }{\partial x_p}-\frac{\partial h_1}{\partial x_p}\otimes l_{m_1}+
  e^{\ad \overline h_1}\left(\frac{\partial h}{\partial x_p}\otimes l_{m_1}+\sum_{i=m_1+1}^n b_{pi}
   \otimes l_i\right), \ p=1, \ldots n .$$
 Since  $\left[\overline h_1, \frac{\partial h}{\partial x_p}\otimes
 l_{m_1}\right]=0,$ we get
   $e^{\ad \overline h_1}\left( \frac{\partial h}{\partial x_p}\otimes l_{m_1}\right)=
 \frac{\partial h_1}{\partial x_p}\otimes l_{m_1},$ $p=1, \ldots n $. It is easy to see that
 $e^{\ad \overline h_1}\left(\sum_{i=m_1+1}^n b_{pi} \otimes l_i\right) \in R \otimes \langle l_{m_1+1}, \ldots l_n \rangle$,
 because $R \otimes \langle l_{m_1+1}, \ldots l_n \rangle$ is an ideal
 of the algebra $R \otimes L$.
 Then we have
 \[e^{\ad \overline h_1}\left(\frac{\partial }{\partial x_p}+\overline{b_p}\right)=
 \frac{\partial }{\partial x_p}-\frac{\partial h_1}{\partial x_p}\otimes l_{m_1}+\frac{\partial h_1}{\partial x_p}\otimes l_{m_1}+\overline {d_p},\]
 for $\overline {d_p}=e^{\ad \overline h_1}\left(\sum_{i=m_1+1}^n b_{pi} \otimes l_i\right) \in R \otimes \langle l_{m_1+1}, \ldots l_n \rangle.$
  Therefore,
  \[e^{\ad \overline h_1}\left(\frac{\partial }{\partial x_p}+\overline{b_p}\right)=\frac{\partial }{\partial x_p}+\overline {d_p},
  p=1, \ldots n .\]
  Denote by  $d_{pi}$  the coordinates of  the element $\overline d_{p}$ in basis
  $\{1\otimes l_{1}, \ldots , 1\otimes l_{n}\}$, $p=1, \ldots ,
  n,$ i.e.   $\overline {d_p}=\sum_{i=1}^n d_{pi} \otimes l_i$.
  Consider the matrix $D=\left(
                                                                            \begin{array}{ccc}
                                                                              d_{11} & \ldots & d_{1n} \\
                                                                              \vdots & \ddots & \vdots \\
                                                                              d_{n1} & \ldots & d_{nn} \\
                                                                            \end{array}
                                                                          \right)
  $.
  We have just proved  that the first nonzero column of the matrix $D$ has the number $m_2>m_1$.
  Analogously, applying the automorphism
  $e^{\ad \overline h_1}$ to the elements $\frac{\partial }{\partial x_p}+\overline {d_p}$, $p=1, \ldots n $,
 we get the elements $\frac{\partial }{\partial x_p}+\overline {f_p}$.
 Define the elements $f_{ij}$, $i,j=1, \ldots n $ from the  equalities
 $\overline {f_p}=\sum_{i=1}^n f_{pi} \otimes l_i$.
  The first nonzero column in the matrix $F=\left(
                                                                            \begin{array}{ccc}
                                                                              f_{11} & \ldots & f_{1n} \\
                                                                              \vdots & \ddots & \vdots \\
                                                                              f_{n1} & \ldots & f_{nn} \\
                                                                            \end{array}
                                                                          \right)
  $ has the number $m_3>m_2$. It is obvious that after  $s$ steps ($s \leq n$) we get
  the elements $\frac{\partial }{\partial x_p}$, $p=1, \ldots n $
  with the corresponding zero matrix.
  Therefore $e^{\ad \overline h_s} \ldots e^{\ad \overline h_1}\left(\frac{\partial
   }{\partial x_p}+\overline{b_p}\right)=
  \frac{\partial }{\partial x_p}$,  $p=1, \ldots n $.
  Since the Lie algebra $R \otimes L$ is nilpotent, there exists (by Campbell-Baker-Hausdorff
  formula) an element $w \in R \otimes L$ such that
  $e^{\ad \overline h_s} \ldots e^{\ad \overline h_1}=
  e^{\ad -\overline {w}}.$
  Finally, let  $\overline L$  be a basic subalgebra which is
  isomorphic to $L$, $\lbrace D_1,\ldots , D_n \rbrace$ be a basis of $\overline
  L$. Write
  $\frac{\partial }{\partial x_p}=\sum_{j=1}^n {b_{pj}D_j}$,
  the isomorphism is defined by the map $l_i \mapsto D_i$. Then
  $\overline{b_p}=\sum_{i=1}^n {b_{pi}\otimes l_i}$, satisfies (1),
  therefore $\overline{b_p}=\sum_{i=1}^{\infty}{\frac{1}{i!} \left({
  {\ad}}\
  \overline w\right)^{i-1}\left(\frac{\partial \overline w}{\partial x_p}\right)}$
  for an element $\overline w \in R\otimes _{\mathbb K}L$ such that  $\det J(\overline w)$ is invertible.
  This completes the proof of the theorem.
\end{proof}

\begin{example}~\label{geizen}
Let $L=H_n$ be the $2n+1$-dimensional Heisenberg Lie algebra, $\{
l_1, \ldots , l_{2n+1}\}$ be its standard  basis with
multiplication rule $\left[l_i,l_{n+i}\right]=l_{2n+1}$ for $1\leq
i\leq n,$ (other products are zero). Then, in this basis
$c_{i,n+i}^{2n+1}=1$, $c_{n+i,i}^{2n+1}=-1$, $1\leq i\leq n$,
 other  structure constants are zero. Take $\overline
w=\sum_{i=1}^n{-x_i \otimes l_i}$. It is clear that  $\frac
{\partial \overline w}{\partial x_p}=-1 \otimes l_p$ and
\begin{equation}
  \overline b_p=1 \otimes l_p-\frac{1}{2}\left[\sum_{i=1}^n{x_i \otimes l_i},
  1 \otimes l_p\right], 1 \leq p \leq n \label{geiz}
  \end{equation}
   in the Lie algebra $\widehat L$. Easy calculation shows that
   $\overline b_p=1 \otimes l_p+\frac{1}{2}x_{p+n}\otimes l_{2n+1}$, $p\leq n$,
   $\overline b_p=1 \otimes l_p-\frac{1}{2}x_{p-n}\otimes l_{2n+1}$,
   $n<p\leq 2n$ and  $\overline b_{2n+1}=1 \otimes l_{2n+1}$.  It can be easily shown that
    $\det(b_{ij})=1$. Passing to the  inverse matrix $B^{-1}$ to the matrix
     $B=(b_{pi})_{p,i=1}^n$ one can easily show  that the linear span over $\mathbb K$  of
   the following derivations is a basic subalgebra of $W_{2n+1}(\mathbb{K})$ which is  isomorphic to $H_n$:

  $$ D_i=\frac {\partial }{\partial x_i}-\frac {1}{2}x_{n+i}\frac {\partial }{\partial x_{2n+1}}, 1\leq i\leq
  n,$$
  $$   D_i=\frac {\partial }{\partial x_i}+\frac {1}{2}x_{n-i}\frac {\partial }{\partial x_{2n+1}}, n<i\leq 2n,
  \  \ \  D_{2n+1}=\frac {\partial }{\partial x_{2n+1}}$$.
\end{example}

\section{On the solvable basic Lie subalgebras}
Some known properties of finite dimensional Lie algebras and
modules over them are collected in the next Lemma.

\begin{lemma} \label{decom}
Let $L$ be a finite dimensional Lie algebra over an algebraically
closed field of zero characteristic and let $H$ be  any its Cartan
subalgebra. Then

{\rm (1)}  if the algebra  $L$ is solvable, then  $L=H+[L,L];$

{\rm (2)}  if $L$ is semisimple and $L= N_{-}\oplus H\oplus N_{+}$
is  its triangular decomposition, then the subalgebras $N_{+}$ and
$N_{-}$ act  nilpotently  on every finite dimensional module $M$
over the Lie algebra  $L.$

\end{lemma}

\begin{proposition}\label{sol}
Let $L$ be an arbitrary $n$-dimensional ($n\geq 1$) solvable
Lie algebra over an algebraically closed field $\mathbb{K}$
of zero characteristic, then there is a basic subalgebra $\overline{L}$
of $W_n(\mathbb{K})$, such that $\overline{L}$ is isomorphic to $L$.
\end{proposition}
\begin{proof}
Let $H$ be any Cartan subalgebra of $L$. Take a basis  $\lbrace
l_1,\ldots ,l_n \rbrace$  of $L$ with the following property:
 $\lbrace l_1,\ldots ,l_m \rbrace$  is a basis of $H$, and
if $H\cap [L, L]\not= 0,$ then $\lbrace l_{k+1},\ldots ,l_m
\rbrace$ is a basis of $H\cap [L, L]$, $m \geq k$,
  $\lbrace l_{k+1},\ldots ,l_n \rbrace$ is a basis of $[L, L]$.
Note that $[L, L]$ is a nilpotent ideal of $L$ because $L$ is
solvable and the field  $\mathbb{K}$ has zero characteristic; the
subalgebra $H$ is nilpotent as a Cartan subalgebra of $L.$

Now, put $\overline{w}=\sum_{i=1}^k {x_i \otimes l_i}$. Consider
the linear map $\phi=\sum_{i=1}^{\infty}{\frac{1}{i!}(\ad
\overline{w})^{i-1}}$
 from the $R$-module $R \otimes _{\mathbb K}H$ to itself
(since $H$ is nilpotent, the sum is finite). By Lemma
\ref{properties}, $\phi$ is an automorphism of the $R$-module $R
\otimes _{\mathbb K}H$. As $R \otimes _{\mathbb K}(H\cap [L, L])$
is an ideal of the algebra $R \otimes _{\mathbb K}H$, the map $\ad
\overline{w}$ saves the submodule $R \otimes _{\mathbb K}(H\cap
[L, L])$. The set  of elements $\lbrace -1\otimes l_{i}\rbrace ,
 i=1, \ldots ,m $ is a basis of the $R$-module $R\otimes
_{\mathbb K}H$, so it is obvious that  $\lbrace \phi (-1\otimes
l_{i})\rbrace, i=1, \ldots ,m$ is also a basis of $R\otimes
_{\mathbb K}H.$ Further, the set $\lbrace \phi (-1\otimes l_{i})
\rbrace, i=k+1, \ldots ,m$ is a basis of the submodule $R\otimes
_{\mathbb K}(H\cap [L, L])$. Then, it is clear that the elements
$\phi (-1\otimes l_{i}), i=1, \ldots ,k$ and $-1\otimes l_{i},
i=k+1, \ldots , m$ together form a basis of the $R$-module
$R\otimes _{\mathbb K}H.$

Put $ d_{p}=\phi (-1\otimes l_{i}), i=1, \ldots , k$ and consider
the  automorphism $\theta=\exp \ad w$ of the algebra
$\widehat{H}$. Then we get for $p=1, \ldots k$:
$$\theta\left(\frac{\partial}{\partial
x_p}\right)=\frac{\partial}{\partial
x_p}+\left[w,\frac{\partial}{\partial x_p}\right]
+\frac{1}{2!}\left[w,\left[w,\frac{\partial}{\partial
x_p}\right]\right]+ \ldots =$$ $$= \frac{\partial}{\partial
x_p}-\frac{\partial w}{\partial
x_p}-\frac{1}{2!}\left[w,\frac{\partial w}{\partial x_p}+
\ldots\right]= \frac{\partial}{\partial x_p}+\phi(-1 \otimes
l_p)=\frac{\partial}{\partial x_p}+\overline{d_p}.$$ It follows
from this that

  $\left[\frac{\partial}{\partial x_p}+\overline{d_p},\frac{\partial}{\partial x_q}+\overline{d_q}\right]=
 \left[\theta\left(\frac{\partial}{\partial x_p}\right),\theta\left(\frac{\partial}{\partial x_q}\right)\right]=
 \theta\left( \left[\frac{\partial}{\partial x_p},\frac{\partial}{\partial x_q}\right]\right)=0$, $p,q=1, \ldots k$.
 Consider the elements $\overline{d_p}$, $p=1, \ldots k$ as elements of the algebra $R \otimes L$.
Put $\overline{d_p}=0$, $p=k+1, \ldots n$. It is clear, taking
into account the choice of $\overline{w}$, that  $\frac{\partial
\overline{d_p}}{\partial x_q}=0$, $q=k+1, \ldots n$. Therefore,
$\left[\frac{\partial}{\partial
x_p}+\overline{d_p},\frac{\partial}{\partial
x_q}+\overline{d_q}\right]=0$, $p,q=1, \ldots n$.

Now put $\overline{u}=\sum_{i=k+1}^n {x_i \otimes l_i}$. Consider
the linear map $\psi$ from the $R$-module $ R \otimes L$ to
$R\otimes L$: $\psi=\sum_{i=1}^{\infty}{\frac{1}{i!}(\ad
\overline{u})^{i-1}}$. (The sum is finite, because $R \otimes
[L,L]$ acts nilpotently on $R \otimes L$  and $u \in R \otimes
[L,L]$). By Lemma \ref{properties}, $\psi$ is an automorphism of
the $R$-module $R \otimes L$. It is obvious that the elements
$\overline{d_p}$, $p=1, \ldots k$ and $-1 \otimes l_i$, $i=k+1,
\ldots n$ form together  a basis of the $R$-module $R \otimes L$.
Note that the set $\lbrace -1 \otimes l_i \rbrace$, $i=k+1, \ldots
n$ is a basis of the $R$-module $R \otimes [L,L]$. Therefore the
elements $\psi(-1 \otimes l_i)$, $i=k+1, \ldots n$ form a basis of
the $R$-module $R \otimes [L,L]$.

Consider now the automorphism $\eta=e^{\ad \overline{u}}=
\sum_{i=0}^{\infty}{\frac{1}{i!}(\ad \overline{u})^{i}}$ of the
algebra $\widehat L.$ Then, the elements $ \eta(\overline{d_p}) $,
$p=1, \ldots k$ and $\eta(-1 \otimes l_p)$, $p=k+1, \ldots n$ form
a basis of the $R$-module $R \otimes L$, and the set $\eta(-1
\otimes l_p)$, $p=k+1, \ldots n$ is a basis of the $R$-submodule
$R \otimes [L,L]$. Note, that $\psi(-1 \otimes l_p)$, $p=k+1,
\ldots n$ is also a basis of the $R$-submodule $R \otimes [L,L]$.
Then, the  elements $\eta(\overline{d_p})$, $p=1, \ldots k$ and
$\psi(-1 \otimes l_p)$, $p=k+1, \ldots n$ form together a basis of
the $R$-module $R \otimes L$. Put
$\overline{b_p}=\eta(\overline{d_p})$, $p=1, \ldots k$ and
$\overline{b_p}=\psi(-1 \otimes l_p)$, $p=k+1, \ldots n$. Writing
down $\overline{b_p}=\sum_{i=1}^n(b_{pi} \otimes l_i)$, we have
$\det(b_{pi})_{p,i=1}^{n} \in \mathbb{K}^{\star}$.

By construction of the element $\overline{u}$, it holds
$\left[\overline{u},\frac{\partial}{\partial x_p}\right]=0$  for
$p=1, \ldots k$, so we have  $e^{\ad
\overline{u}}\left(\frac{\partial}{\partial
x_p}\right)=\frac{\partial}{\partial x_p}$, $p=1, \ldots k$. Note
that
\[\frac{\partial}{\partial x_p}+\overline{b_p}=\frac{\partial}{\partial x_p}+\eta(\overline{d_p})
=e^{\ad \overline{u}} \left(\frac{\partial}{\partial x_p}
\right)+e^{\ad \overline{u}}(\overline{d_p})= e^{\ad \overline{u}}
\left(\frac{\partial}{\partial x_p} + \overline{d_p}\right)  \
p=1, \ldots k.
\]
It is easy to see, that for $p=k+1, \ldots n$ it holds
\[-1 \otimes l_p=-\frac{\partial
\overline{u}}{\partial x_p}=
\left[\overline{u},\frac{\partial}{\partial x_p}\right].\] Hence,
 the relations
\[\frac{\partial}{\partial x_p}+\overline{b_p}=\frac{\partial}{\partial x_p}+\psi(-1 \otimes l_p)=
\frac{\partial}{\partial x_p}+\left(E+\frac{1}{2!}(\ad
\overline{u})+\frac{1}{3!}(\ad\overline{u})^2+ \ldots\right)(-1
\otimes l_p)=\]
\[\frac{\partial}{\partial x_p}+(-1 \otimes l_p)+\frac{1}{2!}[\overline{u},-1 \otimes l_p]+
\frac{1}{3!}[\overline{u},[\overline{u},-1 \otimes l_p]]+
\ldots=\]
\[=\frac{\partial}{\partial x_p}+\left[\overline{u},\frac{\partial}{\partial x_p}\right]+
\frac{1}{2!}\left[\overline{u},\left[\overline{u},\frac{\partial}{\partial
x_p}\right]\right]+\ldots= e^{\ad
\overline{u}}\left(\frac{\partial}{\partial x_p}\right).\] hold
for $p=k+1, \ldots n$.
 Since $\overline{d_p}=0$ for $p=k+1, \ldots n$, we get
\[\frac{\partial}{\partial x_p}+\overline{b_p}=e^{\ad \overline{u}}\left(\frac{\partial}{\partial x_p}+\overline{d_p}\right).\]
Thus, $$\left[\frac{\partial}{\partial
x_p}+\overline{b_p},\frac{\partial}{\partial
x_q}+\overline{b_q}\right]= \left[e^{\ad
\overline{u}}\left(\frac{\partial}{\partial
x_p}+\overline{d_p}\right), e^{\ad
\overline{u}}\left(\frac{\partial}{\partial
x_q}+\overline{d_q}\right)\right] =$$ $$ =e^{\ad
\overline{u}}\left[\frac{\partial}{\partial
x_p}+\overline{d_p},\frac{\partial}{\partial
x_q}+\overline{d_q}\right]=0.$$ Therefore, by Lemma \ref{main}
there exists a basic subalgebra of $W_n{\mathbb(K)}$ which is
isomorphic to $L$.

\end{proof}

\section {The main theorem}
\begin{lemma}~\label{gen}
Let $L$ be $n$-dimensional Lie algebra over an algebraically
closed field $\mathbb K$ of zero characteristic. Let $L=L_1+L_2$,
where $L_1$, $L_2$ are  subalgebras of $L$ such  that  $L_1\cap
L_2= \lbrace 0 \rbrace$, $\dim L_1=m<n$. Assume that the
subalgebra $L_2$ acts nilpotently (by means of multiplication) on
$L$, that is $(\ad L_2)^k (L)=0$ for some $k.$ If there exists a
basic subalgebra $\overline{L_1}$ of $W_m(\mathbb{K})$ such that
$\overline{L_1}$ is  isomorphic to $L_1,$ then there exists a
basic subalgebra of $W_n(\mathbb{K})$ such that $\overline{L}$ is
isomorphic to the Lie algebra $L$.
\end{lemma}
\begin{proof}
Take a basis $\{ l_1,\ldots ,l_m\}$ of the subalgebra  $L_1$ and a
basis $\{ l_{m+1},\ldots ,l_n\}$ of the subalgebra $L_2$. Then the
 elements $l_1,\ldots ,l_n$ form a basis of $L$.
 Denote by $R_{1}$ the subring $\mathbb K[x_{1},
\ldots , x_{m}]$  of the polynomial ring $R=\mathbb K[x_{1},
\ldots , x_{m}, x_{m+1}, \ldots , x_{n}]$. Since  there exists a
basic subalgebra $\overline{L_1}$ of $W_m(\mathbb{K})$ such that
$\overline{L_1}$ is  isomorphic to $L_{1}$,
 by Lemma (\ref{main}) there exist elements
$\overline{d_p}=\sum_{i=1}^{m} d_{pi}\otimes l_{i} \in
R_{1}\otimes _{\mathbb K}L_{1}, \ p=1, \ldots , m$
  such that for the polynomials $d_{pi}$  the relation (\ref{mainequ}) holds with
  the structure constants of the algebra $L_{1}$ and
  $\det (d_{pi})_{p,i=1}^m \in \mathbb K^\star$.
  As  $L_1$ is a subalgebra of $L$,
  we have a natural embedding of the Lie algebra $R_1 \otimes L_1$ into $R \otimes L$
  (an element $\overline{d_p}=\sum_{i=1}^{m}
d_{pi}\otimes l_{i} \in R_{1}\otimes _{\mathbb K}L_{1}$  maps to
the element $\overline{d_p}=\sum_{i=1}^{n} d_{pi}\otimes l_{i} \in
R\otimes _{\mathbb K}L$, $d_{pi}=0$ for $i=m+1, \ldots n$). Put
$\overline{d_p}=0$ for $p=m+1, \ldots n$. Using  Remark
\ref{equiv}, it is easy to see  that the following relations hold
in  $\widehat{L}$:
\begin{equation}\label{maingen}
\left[\frac{\partial}{\partial
x_p}+\overline{d_p},\frac{\partial}{\partial
x_q}+\overline{d_q}\right]=0, \  p,q=1,\ldots n.
\end{equation}

Put $\overline{b}=\sum_{i=m+1}^n x_i \otimes l_i \in R \otimes
L_2$. By the assumption for the subalgebra $L_2$, the derivation
$\ad \overline{b}$ of the Lie algebra $R \otimes L$ is nilpotent
and therefore the automorphism $\theta =\exp(\ad \overline{b})$
 of the Lie algebra $R \otimes L$  (and $\widehat{L}$) is well defined.
 Denote $\overline{b_p}=-\frac{\partial}{\partial x_p}+\theta\left(\frac{\partial}{\partial x_p}+\overline{d_p}\right)$, $p=1 \ldots n$.
Then $\frac{\partial}{\partial x_p}+
 \overline{b_p}=\theta\left(\frac{\partial}{\partial x_p}+\overline{d_p}\right)$,
and  the following equalities  hold:
 \[\left[\frac{\partial}{\partial x_p}+\overline{b_p},\frac{\partial}{\partial x_q}+\overline{b_q}\right]=
 \left[\theta\left(\frac{\partial}{\partial x_p}+\overline{d_p}\right),
 \theta\left(\frac{\partial}{\partial x_q}+\overline{d_q}\right)\right]=\]
 \[=\theta \left(\left[\frac{\partial}{\partial x_p}+\overline{d_p},
 \frac{\partial}{\partial x_q}+\overline{d_q}\right]\right)=0, \ p,q=1 \ldots n.\]
 Let us show that the set of the elements $\overline{b_p}$, $p=1, \ldots n$ is a basis
 of the free $R$-module $R \otimes L$.
 It holds
  \[\overline{b_p}=-\frac{\partial}{\partial x_p}+\theta\left(
  \frac{\partial}{\partial x_p}+\overline{d_p}\right)=
  -\frac{\partial}{\partial x_p}+\theta\left(\frac{\partial}{\partial x_p}\right)+\theta\left(
  \overline{d_p}\right) \ p=1, \ldots m\]
  and then
  \[\theta\left(\frac{\partial}{\partial x_p}\right)=
  \left(E+\ad \overline{b}+\frac{1}{2!}(\ad \overline{b})^2 + \ldots\right)\left(\frac{\partial}{\partial x_p} \right)=\]
  \[=\frac{\partial}{\partial x_p}+\left[\overline{b},\frac{\partial}{\partial x_p}\right]
  +\left[\overline{b},\left[\overline{b},\frac{\partial}{\partial x_p}\right]\right]+ \ldots=\frac{\partial}{\partial x_p},\]
  since $\left[\overline{b},\frac{\partial}{\partial x_p}\right]=-\frac{
  \partial \overline{b}}{\partial x_p}=0$ for  $p=1, \ldots ,m$ .
  Now consider the elements $\overline {b_p}$ for
   $p=m+1, \ldots n$. In this case $\overline {d_p}=0$. Therefore
  \[\overline {b_p}=-\frac{\partial}{\partial x_p}+\theta\left(\frac{\partial}{\partial x_p}+\overline{d_p}\right)=
  -\frac{\partial}{\partial x_p}+\left(E+\ad \overline{b}+\frac{1}{2!}(\ad \overline{b})^2 + \ldots\right)\left(\frac{\partial}{\partial x_p} \right)=\]
  \[=-\frac{\partial}{\partial x_p}+\frac{\partial}{\partial x_p}+\left[\overline{b},\frac{\partial}{\partial x_p}\right]
  +\left[\overline{b},\left[\overline{b},\frac{\partial}{\partial x_p}\right]\right]+ \ldots=\]
  \[=-\frac{\partial b}{\partial x_p}-\frac{1}{2!}\left[\overline{b},\frac{\partial \overline{b}}{\partial x_p}\right]
  -\frac{1}{3!}\left[\overline{b},\left[\overline{b},\frac{\partial \overline{b}}{\partial x_p}\right]\right]=
  \sum_{i=1}^{\infty}{\frac{1}{i!}(\ad \overline{b})^{i-1}\left(-\frac{\partial \overline{b}}{\partial x_p}\right)}.\]
  (Because of nilpotency of $\ad \overline{b}$, the number of nonzero summands in this
  series is finite).

  Denote $\phi=\sum_{i=1}^{\infty}{\frac{1}{i!}(\ad \overline{b})^{i-1}}$.
  It is easy to see that $\phi$ is an automorphism of the
  free $R$-module $R \otimes L$ (see Lemma \ref{properties}).
   Since $\overline{b}=\sum_{i=m+1}^n x_i \otimes l_i \in R \otimes L_2$,   the $R$-module
  $ R \otimes L_2$ is invariant under action of $\phi$.

  The set of  elements $\overline{d_p}$, $p=1, \ldots m$
   and $-\frac{\partial \overline{b}}{\partial x_p}=-1 \otimes l_p$, \ $p=m+1, \ldots n$
 is a basis of the free $R$-module $ R \otimes L$
 (because this module is the direct sum of the
 $R$-modules $ R \otimes L_1$ and $ R \otimes L_2$). Then,  the
  elements $\theta(\overline{d_p})$
 , $p=1, \ldots m$ and $\theta\left(-\frac{\partial \overline{b}}{\partial x_p}\right)$, $p=m+1, \ldots n$
form a basis of $ R \otimes L$. Since $\phi$ is an automorphism of
the free $R$-module
 $R \otimes L_2$, the set of elements
 $\phi\left(-\frac{\partial \overline{b}}{\partial x_p}\right)$, $p=m+1, \ldots n$ is a basis of this submodule
 (note that the  set of  elements $\theta\left(-\frac{\partial \overline{b}}{\partial x_p}\right)$, $p=m+1,
 \ldots n$  is also a basis of $ R \otimes L_2$).
 It follows from this  that   the elements $\theta(\overline{d_p})$
 , $p=1, \ldots m$ and $\phi\left(-\frac{\partial \overline{b}}{\partial x_p}\right)$, $p=m+1, \ldots n$
 together form a basis of the free $R$-module $ R \otimes L$. Then, using the equalities
 $\overline{b_p}=\theta(\overline{d_p})$
 , $p=1, \ldots m$ and $\overline{b_p}=\phi\left(-\frac{\partial
 \overline{b}}{\partial x_p}\right)$, $p=m+1, \ldots n$,
 we see that $\lbrace \overline{b_p}\rbrace$
 , $p=1, \ldots n$ is a basis of the free $R$-module $ R \otimes
 L$, and therefore  $\det(b_{pi})_{p,i=1}^n \in \mathbb{K}^{\star}$. Hence, by Lemma \ref{main}
there exists a basic subalgebra $\widehat{L}$ of $W_n(\mathbb{K})$
such that $\widehat{L}$ is isomorphic to the Lie algebra $L$.

\end{proof}

 Now we can prove the main theorem of this paper.
\begin{theorem}\label{mainth}
  Let $L$ be any  Lie algebra  $L$ over $\mathbb{K}$ of dimension $n\geq 1$. Then  there
  exists   a basic subalgebra $\overline L$ of $W_n(\mathbb{K})$ such that $\overline{L}$ is  isomorphic to $L$.
\end{theorem}

\begin{proof}
By Proposition \ref{sol} we may assume that $L$ is not solvable.
Let $S=S(L)$ be the solvable radical of $L$, and $L=L_0
\rightthreetimes S$
 be the Levi decomposition of $L$, where $L_0$ is a semisimple subalgebra of $L$.
Let $H_{0}\subseteq L_{0}$ be a Cartan subalgebra of $L_{0}$ and
let  $L=N_{-}\oplus H_{0}\oplus N_{+}$ be the corresponding
triangular decomposition for some choice of simple roots.  Denote
by $B_0$ the Borel subalgebra $B_0=H_0+N_{+}$ of $L_0$. Then the
subalgebra $L_1=S+B_0$ is solvable, therefore by Proposition
\ref{sol} there exists a basic subalgebra of $W_k(\mathbb{K})$,
which is isomorphic to $\overline{L_1}$, where $k=\dim L_1$. Since
$L=L_1+N_{-}$, $L_1 \cap N_{-}=\lbrace 0 \rbrace$, and the
subalgebra $N_{-}$ acts nilpotently (by multiplication) on the
$L_0$-module $L$  (see Lemma \ref{decom}), there exists  by Lemma
~\ref{gen} a basic subalgebra $\overline{L}$ of
     $W_n(\mathbb{K})$ such that $\overline{L}$ is  isomorphic to the Lie algebra $L$.
\end{proof}
\begin{example}
Let $L=sl_{2}(\mathbb K )$ and $\{ E, H, F\}$ be its standard
basis over $\mathbb K.$  We shall construct a basic subalgebra of
$W_{3}(\mathbb K)$, which is isomorphic to $L.$ As the Cartan
subalgebra $\langle H \rangle$ of $L$ is one-dimensional, the
basic subalgebra $\langle -\frac{\partial}{\partial x_{1}}\rangle$
of $W_{1}(\mathbb K)$ is isomorphic to $\langle H \rangle$. To the
element $H\in L$ corresponds the element $1\otimes H$  in the Lie
algebra $\widehat L$ (see the Remark \ref{equiv}). The subalgebra
$N_{+}$ has obviously generators $E$ and $\left [H,E\right ]=2E$.
Put $w=x_2\otimes E$. Therefore, we have
\[e^{\ad (x_2 \otimes E)}\left(\frac{\partial }{\partial x_1}-1 \otimes H\right)
=\frac{\partial }{\partial x_1}+2x_2\otimes E - 1 \otimes
H\]
\[e^{\ad x_2 \otimes E}\left(\frac{\partial }{\partial x_2}\right)=\frac{\partial
}{\partial x_2}-1 \otimes E.\] Further, $N_-$ has  generators $F$
and $\left[H,F\right]=-2F$, $\left[E,F\right]=H$. Analogously we
obtain
\[e^{\ad (x_3 \otimes F)}\left(\frac{\partial }{\partial x_1}+2x_2\otimes E - 1 \otimes H\right)=\]
\[=\frac{\partial }{\partial x_1}+2x_2 \otimes E - (1+2x_2x_3)\otimes H-(2x_3+2x_2x_3^2)\otimes F\]
\[e^{\ad x_3 \otimes F}\left(\frac{\partial }{\partial x_2}-1 \otimes E\right)
=\frac{\partial }{\partial x_2}-1 \otimes E+x_3 \otimes
H+x_3^2 \otimes F\]
\[e^{\ad x_3 \otimes F}\left(\frac{\partial }{\partial x_3}\right)=\frac{\partial }{\partial x_3}-1 \otimes F.\]
 Then, using the inverse matrix
 to the matrix of these elements, we
 get a basis of $W_3(\mathbb{K})$. The linear span of this basis over $\mathbb K$
  is isomorphic to the Lie algebra $L=sl_{2}(\mathbb K )$:
 \[E=-x_3 \frac{\partial}{\partial x_1}+ \left(1+2x_2x_3\right)\frac{\partial}{\partial
 x_2}-x_3^2 \frac {\partial}{\partial x_3}\]
 \[H=\frac{\partial}{\partial x_1}-2x_2 \frac{\partial}{\partial x_2}+2x_3
 \frac{\partial}{\partial x_3}, \ \   F=\frac{\partial}{\partial x_3}. \]
\end{example}

\begin{remark}

Since the set of all linear homogeneous derivations of
$W_n(\mathbb{K})$ form a Lie algebra, which is isomorphic to
$gl_{n}(\mathbb K)$, there are embeddings of any $n$-dimensional
Lie algebra $L$ without center into $W_n(\mathbb{K})$.
 But for the image $\overline L$ of $L$ by such
an embedding in $W_n(\mathbb{K})$ and a basis $\{ D_{1}, \ldots ,
D_{n}\}$ of $\overline L$  such that $D_{i}=\sum
_{j=1}^{n}f_{ij}\frac{\partial}{\partial x_j}$, the determinant
$\det (f_{ij})$ is not a constant. Therefore $\overline L$ is not
a basic Lie subalgebra of $W_n(\mathbb{K}).$
\end{remark}

The author is grateful to Professor A.P.Petravchuk for suggesting
the problem and for constant attention to this work.

\end{document}